\documentclass{amsart}

\usepackage{amssymb}

\newcommand{\CC}{\mathbb{C}}
\newcommand{\RR}{\mathbb{R}}
\newcommand{\TT}{\mathbb{T}}

\newcommand{\NN}{\mathbb{N}}
\newcommand{\ZZ}{\mathbb{Z}}

\newcommand{\cF}{\mathcal{F}}
\newcommand{\cZ}{\mathcal{Z}}

\newcommand{\Gram}{\textrm{Gram}}
\newcommand{\Toep}{\textrm{Toep}}

\newcommand{\real}{\textrm{Re}}

\newcommand{\Mon}{\textrm{Mon}}
\newcommand{\Trace}{\textrm{Trace}}

\newtheorem{thm}{Theorem}
\newtheorem{lem}[thm]{Lemma}
\newtheorem{prop}[thm]{Proposition}
\newtheorem{cor}[thm]{Corollary}
\newtheorem{defn}[thm]{Definition}
\newtheorem{ex}[thm]{Example}

\newtheorem{prob}{Problem}

\begin{document}

\title{Hermitian Sums of Squares Modulo Hermitian Ideals}
\author{Glen Frost}
\maketitle

\begin{abstract}
In this work we study the problem of writing a Hermitian polynomial as a Hermitian sum of squares modulo a Hermitian ideal.
We investigate a novel idea of Putinar-Scheiderer to obtain necessary matrix positivity conditions for Hermitian polynomials to be Hermitian sums of squares modulo Hermitian ideals.
We show that the conditions are sufficient for a class of examples making a connection to the operator-valued Riesz-Fejer theorem and block Toeplitz forms.
The work fits into the larger themes of Hermitian versions of Hilbert's 17-th problem and characterizations of positivity.
\end{abstract}

\tableofcontents

\section{Introduction}

In this paper we study the problem of writing a Hermitian polynomial as a Hermitian sum of squares modulo the Hermitian ideal $(z^N \bar{z}^N -1)$, where $N$ is a positive integer.
We investigate an idea of Putinar and Scheiderer from \cite{MR2729629} which constructs a counterexample to a question of D'Angelo.
The idea provides necessary conditions on the Hermitian polynomial, and in our situation we show that they are sufficient.
We begin by stating the problem, discussing the context, describing the idea to obtain the matrix positivity conditions, and formulating the main result.

\subsection{Problem Statement}

Let $\CC[z,\bar{z}]$ denote the polynomial algebra in the variables $z=(z_1,\dots,z_n)$ and $\bar{z}=(\bar{z}_1,\dots,\bar{z}_n)$ with coefficients in $\CC$.
We can write each element $f \in \CC[z,\bar{z}]$ using standard multinomial notation:
\begin{equation}
  \label{eq:1}
  f(z,\bar{z}) = \sum_{\alpha,\beta } a_{\alpha \beta} \bar{z}^\alpha z^\beta
\end{equation}
Or in dimension one ($n=1$) using matrix notation:
\begin{equation}
  \label{eq:3}
  f(z,\bar{z}) =
  \begin{bmatrix}
    1 \\ z \\ \vdots \\ z^d
  \end{bmatrix}^\ast
  \begin{bmatrix}
    a_{00} & a_{01} & \cdots & a_{0d} \\
    a_{10} & a_{11} & \cdots & a_{1d} \\
    \vdots & \vdots & \ddots & \vdots \\
    a_{d0} & a_{d1} & \cdots & a_{dd}
  \end{bmatrix}
  \begin{bmatrix}
    1 \\ z \\ \vdots \\ z^d
  \end{bmatrix}
  = \psi_d(z)^\ast A \psi_d(z)
\end{equation}
where
\begin{equation}
  \label{eq:4}
  \psi_d(z) =
  \begin{bmatrix}
    1 & z & \cdots & z^d
  \end{bmatrix}^T
\end{equation}
denotes the \emph{tautological monomial map} and $A=[a_{jk}]$ is an $(d+1)\times (d+1)$ matrix with complex entries.
We often write $\psi(z)$ for $\psi_d(z)$ when $d$ is understood.

Here we use standard matrix notation and operations.  If $A$ is an $n \times m$ matrix, then $A^T$ denotes the transpose, $\overline{A}$ denotes the conjugate, and $A^*$ denotes the conjugate transpose.  We think of elements $v \in \CC^n$ as column vectors and $\langle v , w \rangle = v^\ast w$ denotes the standard inner product on complex Euclidean space, with the convention being conjugate linear in the first component.

A matrix is \emph{Hermitian} if $A=A^\ast$ and positive semidefinite if $v^\ast A v \geq 0$ for all vectors $v$.  We use the term \emph{positive matrix} when $A$ is positive semidefinite and write $A \geq 0$.

We have an involution $f \mapsto f^\ast$ given by conjugation:
\begin{equation}
  \label{eq:2}
  f^\ast(z,\bar{z}) := \overline{f(z,\bar{z})} = \psi(z)^\ast A^\ast \psi(z)
\end{equation}
Say $f$ is \emph{Hermitian} if $f=f^\ast$.
Hermitian polynomials are real-valued on $\CC^n$.
Let $\CC_h[z,\bar{z}]$ denote the collection of Hermitian polynomials.
An ideal $I$ in $\CC_h[z,\bar{z}]$ is called a \emph{Hermitian ideal}.
If $h(z) \in \CC[z]$ is a holomorphic polynomial, then
\begin{equation}
  \label{eq:5}
  |h(z)|^2 = h(z)\overline{h(z)} \in \CC_h[z,\bar{z}]
\end{equation}
is a \emph{Hermitian square}.
Let $\Sigma^2_h \subset \CC_h[z,\bar{z}]$ denote the collection of finite sums of Hermitian squares.
We have the basic test for a Hermitian polynomial $f$ to be a Hermitian sum of squares: $f \in \Sigma^2_h$ if and only if the Hermitian coefficient matrix is positive.

We now state the main problem:
  \begin{prob}
    Suppose $f$ is a Hermitian polynomial and $I$ is a Hermitian ideal.
    Under what conditions on $f$ and $I$ does there exist an identity:
\begin{equation}
      f(z,\bar{z}) = \sum_{j=1}^\ell |h_j(z)|^2 + g(z,\bar{z})
\end{equation}
    where $h_1,\dots,h_\ell$ are holomorphic polynomials and $g \in I$.
  \end{prob}

  We easily obtain the \emph{trivial necessary condition},  point-wise positivity on the zero set of $I$:
\begin{equation}
    f(p,\bar{p}) \geq 0 \qquad \forall p \in \cZ(I)
\end{equation}
  where $\cZ (I)$ denotes the zero-set of $I$:
\begin{equation}
    \cZ (I) = \{ p \in \CC^n \ | \ g(p,\bar{p}) = 0 \ \forall g \in \cZ( I) \}
\end{equation}

This condition is not sufficient in general.  For example, consider the ideal $I=(0)$ and the polynomial $f(z,\bar{z}) = (z+\bar{z})^2$.
Since $z+\bar{z} = 2\real(z)$, we see $f(p,\bar{p}) \geq 0$ for all $p \in \CC$.
Writing $f$ in matrix notation:
\begin{equation}
  \label{eq:7b}
  (z+\bar{z})^2 =     \begin{bmatrix}
      1 \\ z \\ z^2
    \end{bmatrix}^\ast
    \begin{bmatrix}
      0 & 0 & 1 \\
      0 & 2 & 0 \\
      1 & 0 & 0 
    \end{bmatrix}
    \begin{bmatrix}
      1 \\ z \\ z^2
    \end{bmatrix}
\end{equation}
The coefficient matrix is not positive, hence $f \notin \Sigma^2_h$.

\subsection{Context}

In dimension one, for the ideal $I=(z\bar{z}-1)$, the trivial necessary condition is sufficient.  The idea goes back to the classical Riesz-Fejer lemma concerning positive trigonometric polynomials on the circle \cite{MR1580923} \cite{MR1580922}.
The first multivariable result appears in \cite{MR233770} showing that the trivial necessary condition is sufficient for the odd-dimensional sphere in $\CC^n$.
This theorem was rediscovered in \cite{MR1386836} by D'Angelo and collaborators.
At the 2006 AIM Conference ``CR Complexity Theory'' in Palo Alto, D'Angelo asks the following generalization of Quillen's theorem:

\emph{    If $f$ is a Hermitian polynomial on a pseudoconvex hypersurface, then does $f$ agree with a Hermitian sum of squares along the hypersurface?}

  The answer to this paper lies in the paper \cite{MR2729629}.
  Two main ideas of this paper:
  \begin{enumerate}
  \item An obstruction to the question is introduced and a counterexample is constructed.  The obstruction provides necessary matrix positivity conditions for the Hermitian polynomial $f$ beyond pointwise positivity given by the trivial necessary condition.
  \item A characterization of the Hermitian ideals for which every positive Hermitian polynomial is a Hermitian sum of squares is obtained using the Archimedean Positivstellensatz of real algebra.
  \end{enumerate}

  The first point is explored further in the paper \cite{MR2925474}, and the second point is explored further in the paper \cite{MR3426234}.

  Our goal is to continue the idea of the matrix positivity conditions for Hermitian polynomials, and obtain equivalent characterizations for Hermitian sums of squares modulo Hermitian ideals, thereby going beyond the Archimedean positivstellensatz and pointwise-positivity conditions.

\subsection{Matrix Positivity Conditions for Hermitian Polynomials}
  We describe the idea of matrix positivity conditions used in \cite{MR2729629} to construct the counterexample to D'Angelo's question.  Two remarks are needed before we begin.
  
  First, it is notationally convenient to collect the terms of a Hermitian sum of squares into the components of a holomorphic polynomial map.
  A \emph{holomorphic polynomial map} is a map $h: \CC^n \rightarrow \CC^\ell$ whose components are holomorphic polynomials.
  Every Hermitian sum of squares may be written as the squared norm of a holomorphic polynomial map:
  \begin{equation}
    \label{eq:15}
    |h_1(z)|^2 + \dots + |h_\ell(z)|^2 = \langle h(z), h(z) \rangle
  \end{equation}
  where $h(z) = (h_1(z),\dots,h_\ell(z))$ is a holomorphic polynomial map.

  Second, for a polynomial $f(z,\bar{z})\in \CC[z,\bar{z}]$, we may polarize and treat $z$,$\bar{z}$ as independent variables.  If $f$ is Hermitian, then the polarization satisfies the Hermitian symmetric condition:
  \begin{equation}
    \label{eq:6b}
    f(z,\bar{w}) = \overline{f(w,\bar{z})}
  \end{equation}
  
  Now given a Hermitian polynomial $f$ and a Hermitian ideal $I$, we show how to obtain necessary matrix positivity conditions which must be satisfied by $f$.

  Suppose $f$ is a Hermitian sum of squares modulo $I$:
  \begin{equation}
    \label{eq:16}
    f(z,\bar{z}) = \langle h(z), h(z) \rangle + g(z,\bar{z})
  \end{equation}
  where $h$ is a holomorphic polynomial map and $g \in I$.

  Suppose further that $p_1,\dots,p_\ell \in \CC^n$ are points such that
  \begin{equation}
    \label{eq:17}
    g(p_j, \bar{p}_k) = 0 \qquad \forall g \in I \qquad \forall j,k=1,\dots,\ell
  \end{equation}
  Then
  \begin{equation}
    \label{eq:18}
    f(p_j,\bar{p}_k) = \langle h(p_k), h(p_j) \rangle \qquad \forall j,k=1,\dots,\ell
  \end{equation}
  The matrix determined by the right side is a matrix of pairwise inner products (a \emph{Gram matrix}), which is always positive, hence
  \begin{equation}
    \label{eq:19}
    [  f(p_j,\bar{p}_k)]_{j,k=1}^\ell \geq 0
  \end{equation}
  Thus we conclude:
  \begin{lem}[Matrix Positivity Conditions]\label{MATRIX-POSITIVITY-CONDITIONS}
    Let $f$ be a Hermitian polynomial and $I$ a Hermitian ideal.
    If $f \in \Sigma^2_h+I$, then $f$ satisfies \eqref{eq:19} for all collections of points $p_1,\dots,p_\ell$ satisfying \eqref{eq:17}
  \end{lem}

  Using this idea we obtain necessary conditions for $f \in \Sigma^2_h + I$.
  Our main question of interest is whether these conditions are sufficient?

  With these ideas in mind, given a Hermitian polynomial $f$ and points $p_1,\dots,p_\ell$, we define the \emph{Gram matrix of $f$ on $p_1,\dots,p_\ell$} as the matrix of pairwise evaluations:
  \begin{equation}
    \label{eq:8b}
    \Gram(f)[p_1,\dots,p_\ell] := [ f(p_j,\bar{p_k})]_{j,k=1}^\ell
  \end{equation}
  
  As an example, consider $f=(z+\bar{z})^2$ and $I=(0)$.
  Choose $p_1 = 0$ and $p_2 = 1$.
  Then consider the Gram matrix of $f$ on $p_1,p_2$:
  \begin{equation}
    \label{eq:20}
    \Gram(f)[p_1,p_2] =
    \begin{bmatrix}
      f(p_1,\bar{p}_1) & f(p_1,\bar{p}_2) \\
      f(p_2,\bar{p}_1) & f(p_2,\bar{p}_2)
    \end{bmatrix}
    =
    \begin{bmatrix}
      0 & 1 \\
      1 & 4
    \end{bmatrix}
  \end{equation}
  The matrix has determinant $-1$, hence is not positive.
  Therefore $f \notin \Sigma^2_h$.
  
  \subsection{Main Results}
  
  Our main result is to prove sufficiency of the matrix positivity conditions for the ideal $    I = (z^N \bar{z}^N - 1 )$, where $N$ is a positive integer.
  The conditions are obtained as follows:
  Let $P(z,\bar{z}) = z^N \bar{z}^N -1$ and $\omega = e^{2\pi i / N}$.
  Then:
  \begin{equation}
    \label{eq:22}
    P ( \omega^j \xi , \overline{\omega^k \xi} ) = 0 \qquad \forall \xi \in \TT \qquad \forall j,k=0,\dots,N-1
  \end{equation}
  Therefore by Lemma \ref{MATRIX-POSITIVITY-CONDITIONS}, if $f \in \Sigma^2_h + (z^N \bar{z}^N - 1)$, then
  \begin{equation}
    \label{eq:23}
    [ f(  \omega^j \xi , \overline{\omega^k \xi} ) ]_{j,k=0}^{N-1} \geq 0 \qquad \forall \xi \in \TT
  \end{equation}
  We will show that these conditions are sufficient.
  The case $N=1$ is the classical Riesz-Fejer lemma.

The next example provides evidence to the investigation.

\renewcommand{\labelenumi}{(\roman{enumi})}
\begin{ex} ($N=2$)
We demonstrate an example of a Hermitian polynomial $f$ such that:
  \begin{enumerate}
  \item $f \notin \Sigma^2_h$
  \item $\Gram(f)[e^{i\theta},-e^{i\theta}] \geq 0$ for all $\theta$
  \item $f \in \Sigma^2_h + (z^2\bar{z}^2-1)$
  \end{enumerate}

  Consider:
  \begin{align}
    f(z,\bar{z}) &= 10 + 2z + 2\bar{z} + 10z\bar{z} -2z^2 \bar{z} -2z\bar{z}^2 \\
    &=
      \begin{bmatrix}
        1 \\ z \\ z^2
      \end{bmatrix}^\ast
    \begin{bmatrix}
      10 & 2 & 0 \\
      2 & 10 & -2 \\
      0 & -2 & 0
    \end{bmatrix}
       \begin{bmatrix}
        1 \\ z \\ z^2
      \end{bmatrix}
  \end{align}
  \begin{enumerate}
  \item By Sylvester's criterion, the Hermitian coefficient matrix is not positive, hence $f \notin \Sigma^2_h$.
  \item We compute $\Gram(f)[e^{i\theta},-e^{i\theta}]$:
    \begin{equation}
      \label{eq:3a}
      \Gram(f)[e^{i\theta},-e^{i\theta}] =
      \begin{bmatrix}
        20 & 8i\sin(\theta) \\
        -8i\sin(\theta) & 20 
      \end{bmatrix}      
    \end{equation}
    We have:
    \begin{equation}
      \label{eq:2a}
      \det \Gram (f)[e^{i\theta},-e^{i\theta}] = 400 -64\sin^2(\theta) > 0      
    \end{equation}
    Thus by Sylvester's criterion:
    \begin{equation}
      \label{eq:111}
      \Gram(f)[e^{i\theta},-e^{i\theta}] \geq 0 \qquad \forall \theta \in [0,2\pi]
    \end{equation}

  \item Now
    \begin{equation}
      \label{eq:4a}
      f(z,\bar{z})+ 5(z^2\bar{z}^2-1) =
       \begin{bmatrix}
        1 \\ z \\ z^2
      \end{bmatrix}^\ast
      \begin{bmatrix}
        5 & 2 & 0 \\
        2 & 10 & -2 \\
        0 & -2 & 5
      \end{bmatrix}
 \begin{bmatrix}
        1 \\ z \\ z^2
      \end{bmatrix}      
    \end{equation}
    By Sylvester's criterion, the Hermitian coefficient matrix is positive.
    Thus:
    \begin{equation}
      \label{eq:5a}
      f \in \Sigma^2_h + (z^2 \bar{z}^2-1)      
    \end{equation}

  \end{enumerate}
\end{ex}

We now describe the main result.
Consider a Hermitian polynomial $g(z,\bar{z})$ and assume $[g(\omega^j \xi, \overline{\omega^k \xi})]_{j,k=0}^{N-1} \geq 0$ for all $\xi \in \TT$.  We may write $g$ in block matrix form:
\begin{equation}
    g(z,\bar{z}) =
  \begin{bmatrix} \psi(z) \\ z^N \psi(z) \\ \vdots \\ z^{Nm}\psi(z) \end{bmatrix}^\ast
  \begin{bmatrix}
    B_{00} & B_{01} & \cdots & B_{0m} \\
    B_{10} & B_{11} & \cdots & B_{1m} \\
    \vdots & \vdots & \ddots & \vdots \\
    B_{m0} & B_{m1} & \cdots & B_{mm} 
  \end{bmatrix}
    \begin{bmatrix} \psi(z) \\ z^N \psi(z) \\ \vdots \\ z^{Nm}\psi(z) \end{bmatrix}
\end{equation}
Reducing coefficients modulo $(z^N\bar{z}^N-1)$, we have that $g$ is equivalent to a unique polynomial of the form
\begin{equation}
    f(z,\bar{z}) = 
  \begin{bmatrix} \psi(z) \\ z^N \psi(z) \\ \vdots \\ z^{Nm}\psi(z) \end{bmatrix}^\ast
  \begin{bmatrix}
   A_0 & A_1 & \cdots & A_m \\
   A_{-1} & 0 &  \cdots & 0 \\
   \vdots & \vdots & \ddots & \vdots \\
   A_{-m} & 0 & \cdots & 0 
  \end{bmatrix}
  \begin{bmatrix} \psi(z) \\ z^N \psi(z) \\ \vdots \\ z^{Nm}\psi(z) \end{bmatrix}
\end{equation}
without changing the positivity condition.
Our plan is to show that the matrix positivity conditions imply positivity of the associated block Toeplitz form:
\begin{equation}
     \Toep(A_0,\dots,A_m) :=
  \begin{bmatrix}
    A_0 & A_1 & A_2 & \cdots & A_m \\
    A_{-1} & A_0 & A_1 & \ddots & \vdots \\
    A_{-2} & A_{-1} & A_0 & \ddots & \vdots \\
    \vdots & \ddots & \ddots & \ddots & A_1 \\
    A_{-m} & \cdots & \cdots & A_{-1} & A_0 
  \end{bmatrix}
\end{equation}
We then use the block Toeplitz positivity to invoke the operator-valued Riesz-Fejer theorem \cite{MR2743422} to obtain $f \in \Sigma^2_h +(z^N\bar{z}^N-1)$, and hence $g \in \Sigma^2_h + (z^N \bar{z}^N-1)$ as well.  The result is formulated as Theorem \ref{MAIN-THEOREm}.

\subsection{Outline}
  
The paper is organized as follows:
In section 2 we prove some basic lemmas and computations involving orthogonal polynomials on the circle inspired by the $N=1$ case.  We also introduce the block trace parametrization of trigonometric polynomials and recall the operator-valued RF theorem for our purposes.
In section 3 we define a functional $\cF_N$, then discuss basic properties and how to use the functional to represent matrix inner products.
In section 4 we utilize the developed tools to prove the full characterization of Hermitian sums of squares modulo the Hermitian ideal $I=(z^N\bar{z}^N-1)$.


\section{Preliminaries}

In this section we discuss preliminary ideas for the main result: basic orthogonal polynomial computations on the circle, block trace parametrization of trigonometric polynomials, and the operator-valued Riesz-Fejer theorem.

\subsection{Basic Computations on the Circle}

To begin, let $h_j(z) = z^j$ for $j \in \ZZ$.
We have the classical formulae:
  \begin{equation}
    \label{eq:48}
    \frac{1}{2\pi} \int_{-\pi}^\pi h_j(e^{i\theta}) d\theta = \delta_{j0}    
  \end{equation}
and
\begin{equation}
  \label{eq:49}
    \frac{1}{2\pi} \int_{-\pi}^\pi h_j(e^{i\theta}) \overline{h_k(e^{i\theta})} d\theta = \delta_{jk}  
\end{equation}

  where $\delta_{jk}$ is the standard Kronecker symbol:
  \[
    \delta_{jk} =
    \begin{cases}
      1 & \textrm{ if } j=k \\
      0 & \textrm{ if } j\neq k \\
    \end{cases}
  \]

  We also note the identities:
  \begin{equation}
    \label{eq:6a}
    \overline{h_j(e^{i\theta})} = h_{-j}(e^{i\theta})
  \end{equation}
  \begin{equation}
    \label{eq:7a}
    h_j(z)h_k(z) = h_{j+k}(z)
  \end{equation}
  
Given $f \in \CC[z,\bar{z}]$, we can integrate around the circle to obtain the trace of the coefficient matrix:
  
  \begin{lem}\label{cha:main-results:diagonals-lemma}
 Let $f(z,\bar{z}) = \sum_{j,k=0}^m a_{jk} \bar{z}^j z^k \in \CC[z,\bar{z}]$.
 Then:
 \begin{equation}
   \label{eq:50}
  \frac{1}{2\pi} \int_{-\pi}^\pi f( e^{i\theta}, \overline{e^{i\theta}} ) d\theta = \sum_{\ell = 0}^m a_{\ell \ell}   
 \end{equation}

\end{lem}
\proof
\begin{align}
  \frac{1}{2\pi} \int_{-\pi}^\pi f(e^{i\theta}, \overline{e^{i\theta}}) d \theta &=  \frac{1}{2\pi} \int_{-\pi}^\pi ( \sum_{j,k=0}^m a_{jk} e^{-ij\theta}e^{ik\theta}  ) d\theta \\
                                                                                &= \sum_{j,k=0}^m a_{jk} (  \frac{1}{2\pi} \int_{-\pi}^\pi e^{i\theta(k-j)} d\theta ) \\
                                                                                &= \sum_{j,k=0}^m a_{jk} \delta_{jk} \\
  &= \sum_{\ell=0}^m a_{\ell \ell}
\end{align}
\endproof

Consider a Hermitian polynomial $g(z,\bar{z})$ written in matrix notation:
\begin{equation}
  \label{eq:8a}
  g(z,\bar{z}) =
  \begin{bmatrix}
    1 \\ z \\ \vdots \\ z^m
  \end{bmatrix}^\ast
  \begin{bmatrix}
    b_{00} & b_{01} & \cdots & b_{0m} \\
    b_{10} & b_{11} & \cdots & b_{1m} \\
    \vdots & \vdots & \ddots & \vdots \\
    b_{m0} & b_{m1} & \cdots & b_{mm} 
  \end{bmatrix}
   \begin{bmatrix}
    1 \\ z \\ \vdots \\ z^m
  \end{bmatrix}
\end{equation}

By reducing the coefficients modulo $(z\bar{z}-1)$, we see that $g$ is congruent modulo $(z\bar{z}-1)$ to a unique polynomial of the form:
\begin{equation}
  \label{eq:9a}
  f(z,\bar{z}) =
   \begin{bmatrix}
    1 \\ z \\ \vdots \\ z^m
  \end{bmatrix}^\ast
  \begin{bmatrix}
    a_0 & a_1 & \cdots & a_m \\
    a_{-1} & 0 & \cdots & 0 \\
    \vdots & \vdots & \ddots & \vdots \\
    a_{-m} & 0 & \cdots & 0
  \end{bmatrix}
   \begin{bmatrix}
    1 \\ z \\ \vdots \\ z^m
  \end{bmatrix}
\end{equation}
where $a_{-j} = \bar{a}_j$.
A Hermitian polynomial of this form is called a \emph{trigonometric polynomial} with data $(a_0,\dots,a_m)$.

Since $g(z,\bar{z}) = f(z,\bar{z}) + q(z,\bar{z})(z\bar{z}-1)$ for some Hermitian polynomial $q(z,\bar{z})$, we have the following observations:
\begin{enumerate}
\item $f(e^{i\theta},\overline{e^{i\theta}}) \geq 0$ for all $\theta \in [0,2\pi]$ if and only if $g(e^{i\theta},\overline{e^{i\theta}}) \geq 0$ for all $\theta \in [0,2\pi]$ 
\item $f \in \Sigma^2_h + (z\bar{z}-1)$ if and only if $g \in \Sigma^2_h + (z\bar{z}-1)$
\end{enumerate}

It is known \cite[pg. 17]{MR890515} that the condition
\begin{equation}
  \label{eq:10a}
  f(e^{i\theta},\overline{e^{i\theta}}) \geq 0 \qquad \forall \theta \in [0,2\pi]
\end{equation}
is equivalent to positivity of the associated Toeplitz matrix
\begin{equation}
  \label{eq:11a}
  \Toep(a_0,\dots,a_m) :=
  \begin{bmatrix}
    a_0 & a_1 & a_2 & \cdots & a_m \\
    a_{-1} & a_0 & a_1 & \ddots & \vdots \\
    a_{-2} & a_{-1} & a_0 & \ddots & \vdots \\
    \vdots & \ddots & \ddots & \ddots & a_1 \\
    a_{-m} & \cdots & \cdots & a_{-1} & a_0 
  \end{bmatrix}
\end{equation}
A matrix is \emph{Toeplitz} if the entries are constant along the diagonals.
We collect the information in the following lemma:
\begin{lem}\cite[pg. 17]{MR890515}
  Suppose $f$ is a trigonometric polynomial with data $(a_0,\dots,a_m)$.
  Let $w=
  \begin{bmatrix}
    w_0 & \cdots & w_m
  \end{bmatrix}^T \in \CC^{m+1}$.
  Define $w(z) = w_0 + w_1 z + \dots + w_m z^m \in \CC[z]$.
  Then:
  \begin{equation}
    \label{eq:12a}
    \frac{1}{2\pi}\int_{-\pi}^\pi |w(e^{i\theta})|^2 f(e^{i\theta},\overline{e^{i\theta}}) d\theta
    =
    w^\ast \Toep(a_0,\dots,a_m) w
  \end{equation}
  In particular, if $f(e^{i\theta},\overline{e^{i\theta}}) \geq 0$ for all $\theta \in [0,2\pi]$,
  then $\Toep(a_0,\dots,a_m) \geq 0$.
\end{lem}

Now consider a holomorphic polynomial $h(z) = h_0 + h_1z + \dots + h_m z^m \in \CC[z]$.
Then
\begin{equation}
  \label{eq:13a}
  |h(z)|^2 =
  \begin{bmatrix}
    1 \\ z \\ \vdots \\ z^m
  \end{bmatrix}^\ast
  \begin{bmatrix}
    h_0 \bar{h}_0 & h_1\bar{h}_0 & \cdots &  h_m \bar{h}_0 \\
    h_0 \bar{h}_1 & h_1 \bar{h}_1 & \cdots & h_m \bar{h}_1 \\
    \vdots & \vdots & \ddots & \vdots \\
    h_0 \bar{h}_m & h_1 \bar{h}_m & \cdots & h_m \bar{h}_m
  \end{bmatrix}
  \begin{bmatrix}
    1 \\ z \\ \vdots \\ z^m
  \end{bmatrix}
\end{equation}
Suppose further that $f$ is a trigonometric polynomial with data $(a_0,\dots,a_m)$.
Then, after reducing coefficients modulo $(z\bar{z}-1)$, we see that the following conditions are equivalent:
\begin{enumerate}
\item $f(z,\bar{z}) \equiv |h(z)|^2 \bmod (z \bar{z}-1)$
\item $a_k = \sum_{j=k}^m h_j \bar{h}_{j-k}$ for $k=0,\dots,m$
\end{enumerate}

From which we can conclude that the following conditions are equivalent:
\begin{enumerate}
\item $f \in \Sigma^2_h + (z\bar{z}-1)$
\item There exist $h_0,\dots,h_m \in \CC$ such that $a_k = \sum_{j=k}^m h_j \bar{h}_{j-k}$ for $k=0,\dots,m$
\end{enumerate}

These equations characterize the coefficients of the holomorphic polynomial $h(z)$ (see \cite[pg. 22]{MR890515}), and can be solved by the method of spectral factorization.

Our plan for the case $N>1$ is to develop block analogues of these ideas and invoke the operator-valued Riesz-Fejer lemma.


\subsection{Block Trace Parametrization}\label{sec:block-trace-param}

In this section we discuss a block analog of the trace parametrization technique utilized throughout \cite{MR3618747}.

Now consider a Hermitian polynomial $g(z,\bar{z})$.
We may write $g$ in block matrix form:
\begin{equation}
  \label{eq:14a}
  g(z,\bar{z}) =
  \begin{bmatrix} \psi(z) \\ z^N \psi(z) \\ \vdots \\ z^{Nm}\psi(z) \end{bmatrix}^\ast
  \begin{bmatrix}
    B_{00} & B_{01} & \cdots & B_{0m} \\
    B_{10} & B_{11} & \cdots & B_{1m} \\
    \vdots & \vdots & \ddots & \vdots \\
    B_{m0} & B_{m1} & \cdots & B_{mm} 
  \end{bmatrix}
    \begin{bmatrix} \psi(z) \\ z^N \psi(z) \\ \vdots \\ z^{Nm}\psi(z) \end{bmatrix}
\end{equation}
where $B_{jk} \in \CC^{N \times N}$ and $B_{kj} = B_{jk}^\ast$.
By reducing the coefficients modulo $(z^N \bar{z}^N - 1)$ we see that $g$ is congruent modulo $(z^N \bar{z}^N-1)$ to a unique polynomial of the form

\begin{equation}
  \label{eq:7d}
    f(z,\bar{z}) = 
  \begin{bmatrix} \psi(z) \\ z^N \psi(z) \\ \vdots \\ z^{Nm}\psi(z) \end{bmatrix}^\ast
  \begin{bmatrix}
   A_0 & A_1 & \cdots & A_m \\
   A_{-1} & 0 &  \cdots & 0 \\
   \vdots & \vdots & \ddots & \vdots \\
   A_{-m} & 0 & \cdots & 0 
  \end{bmatrix}
  \begin{bmatrix} \psi(z) \\ z^N \psi(z) \\ \vdots \\ z^{Nm}\psi(z) \end{bmatrix}
\end{equation}

where $A_j \in \CC^{N \times N}$, $A_{-j} = A_j^\ast$ and $\psi(z) = \begin{bmatrix} 1 & z & \cdots & z^{N-1} \end{bmatrix}^T$.
 If a Hermitian polynomial $f$ has the above form, then we say $f$ is \emph{trigonometric modulo $(z^N\bar{z}^N-1)$ of degree $m$ with data $(A_0,\dots,A_m)$}.

 Since
 \begin{equation}
   \label{eq:6b}
   g(z,\bar{z}) = f(z,\bar{z}) + q(z,\bar{z})(z^N \bar{z}^N -1)
 \end{equation}
for some Hermitian polynomial $q$, we see that
 \begin{equation}
   \label{eq:15a}
   \Gram(f)[\xi,\omega \xi, \dots, \omega^{N-1} \xi] = \Gram(g)[\xi,\omega \xi, \dots, \omega^{N-1}\xi]
 \end{equation}
 Hence the Gram matrices are simultaneously positive.

 Furthermore, we have that $f \in \Sigma^2_h + (z^N \bar{z}^N -1)$ if and only if $g \in \Sigma^2_h + (z^N\bar{z}^N-1)$.

 If $f$ is trigonometric modulo $(z^N \bar{z}^N -1)$ with data $(A_0,\dots,A_m)$, then our goal is to show that the condition
 \begin{equation}
   \label{eq:16}
   \Gram(f)[\xi,\omega \xi,\dots,\omega^{N-1}\xi] \geq 0 \qquad \forall \xi \in \TT
 \end{equation}
implies positivity of the associated block Toeplitz matrix
\begin{equation}
  \label{eq:17a}
   \Toep(A_0,\dots,A_m) :=
  \begin{bmatrix}
    A_0 & A_1 & A_2 & \cdots & A_m \\
    A_{-1} & A_0 & A_1 & \ddots & \vdots \\
    A_{-2} & A_{-1} & A_0 & \ddots & \vdots \\
    \vdots & \ddots & \ddots & \ddots & A_1 \\
    A_{-m} & \cdots & \cdots & A_{-1} & A_0 
  \end{bmatrix}
\end{equation}
which allows us to apply the operator-valued RF theorem to obtain $f \in \Sigma^2_h + (z^N \bar{z}^N -1)$.

The idea of the trace parametrization depends on the following observation:
Let $T_k$ denote the elementary block Toeplitz matrix with $I$ on the $k$-th diagonal and $0$ elsewhere, where the main diagonal is counted as $k=0$ and positive diagonals count to the right.
Let $\Trace[Q]$ denote the sum of the diagonal blocks of $Q=[Q_{jk}]_{j,k=0}^m$.
Then
\begin{equation}
  \label{eq:21a}
  \Trace[T_{-k}Q] = \sum_{j=k}^m Q_{j-k,j}
\end{equation}
is the sum of the $k$-th diagonal of $Q$.

Suppose $h(z)$ is a holomorphic polynomial of degree $N(m+1)$.
We can write $|h(z)|^2$ in block matrix form:
\begin{equation}
  \label{eq:18a}
  |h(z)|^2 =
  \begin{bmatrix} \psi(z) \\ z^N \psi(z) \\ \vdots \\ z^{Nm}\psi(z) \end{bmatrix}^\ast
  \begin{bmatrix}
    Q_{00} & Q_{01} & \cdots & Q_{0m} \\
    Q_{10} & Q_{11} & \cdots & Q_{1m} \\
    \vdots & \vdots & \ddots & \vdots \\
    Q_{m0} & Q_{m1} & \cdots & Q_{mm}
  \end{bmatrix}
  \begin{bmatrix} \psi(z) \\ z^N \psi(z) \\ \vdots \\ z^{Nm}\psi(z) \end{bmatrix}
\end{equation}
where $Q = [Q_{jk}]_{j,k=0}^m$ is a positive Hermitian block matrix with $Q_{jk} \in \CC^{N \times N}$ and $Q_{kj} = Q_{jk}^\ast$.

Suppose further that $f$ is trigonometric modulo $(z^N \bar{z}^N -1)$ with data $(A_0,\dots,A_m)$.
Then, after by considering the reduction of coefficients modulo $(z^N\bar{z}^N-1)$, we see that the following conditions are equivalent:
\begin{enumerate}
\item $f(z,\bar{z}) \equiv |h(z)|^2 \bmod (z^N \bar{z}^N -1)$
\item $A_k = \Trace[T_{-k} Q]$
\end{enumerate}

From which we conclude that the following conditions are equivalent:
\begin{enumerate}
\item $f \in \Sigma^2_h + (z^N \bar{z}^N -1)$
\item There exists a positive block matrix $Q=[Q_{jk}]_{j,k=0}^m$ such that $A_k = \Trace[T_{-k}Q]$, $k=0,\dots,m$.
\end{enumerate}

\subsection{Operator-Valued RF Theorem}

We recall the operator-valued Riesz-Fejer theorem:
\begin{thm}\cite[Theorem 2.1]{MR2743422}
  Let $A(z) = \sum_{k=-m}^m A_k z^k$ be a Laurent polynomial with matrix coefficients $A_k \in \CC^{N \times N}$.
  The following conditions are equivalent:
  \begin{enumerate}
  \item $A(\xi)\geq 0$ for all $\xi \in \TT$
  \item $\Toep(A_0,\dots,A_m) \geq 0$
  \item There exists $P(z) = P_0 + P_1 z + \dots + P_m z^M$ with matrix coefficients $P_k \in \CC^{N \times N}$ such that $A(z) = P(z)^\ast P(z)$
  \end{enumerate}
\end{thm}

Let $Q_{jk} = P_j^\ast P_k$ and let $Q= [ Q_{jk} ]_{j,k=0}^m$.
If $A(z) = P(z)^\ast P(z)$, then we can equate coefficients to get $A_k = \textrm{Trace}[T_{-k}Q]$.
Thus we get:

\begin{cor}\label{sec:operator-valued-rf:OPERATOR-RF} Let $A_0,\dots,A_m \in \CC^{N \times N}$.
  The following conditions are equivalent:
  \begin{enumerate}
  \item $\Toep(A_0,\dots,A_m) \geq 0$
  \item There exists a positive block matrix $Q = [Q_{jk}]_{j,k=0}^m$ such that $A_k = \Trace[T_{-k}Q]$.
  \end{enumerate}
\end{cor}

With these considerations in mind, starting with $f$ trigonometric modulo $(z^N \bar{z}^N -1)$ with data $(A_0,\dots,A_m)$, we will assume
\begin{equation}
  \label{eq:19a}
  \Gram(f)[\xi, \omega \xi, \dots, \omega^{N-1} \xi] \geq 0 \qquad \forall \xi \in \TT
\end{equation}
and show that this implies
\begin{equation}
  \label{eq:20a}
  \Toep(A_0,\dots,A_m) \geq 0
\end{equation}
from which we can obtain $f \in \Sigma^2_h + (z^N \bar{z}^N -1)$.


\section{The Functional $\cF_N$}

In this section we define a functional $\cF_N : \CC[z,\bar{z}]\rightarrow \CC$ and use it to represent matrix products.
The goal is to transfer the matrix positivity conditions into block Toeplitz positivity conditions.

\subsection{Definition and Basic Properties}
\begin{defn}
 For $f(z,\overline{z}) \in \CC[z,\overline{z}]$, define:
 \begin{equation}
   \label{eq:52}
 \mathcal{F}_N(f) :=
 \frac{1}{2\pi} \int_{-\pi}^\pi [ \frac{1}{N^2} \sum_{j,k=0}^{N-1} f( \omega^j e^{i\theta} , \overline{\omega^k e^{i\theta}} ) ] d\theta   
 \end{equation}

\end{defn}

Observe that the integrand is the average of the entries of $\Gram(f)[e^{i\theta},\dots,\omega^{N-1}e^{i\theta}]$.

We require the following computation for the next proposition.

\begin{lem}\label{cha:main-results:circulant-average-lemma}
  \begin{equation}
    \label{eq:51}
  \frac{1}{N^2} \sum_{j,k=0}^{N-1} \omega^{\ell (j-k)}
  =
  \begin{cases}
   1 & \text{if } \ell \mid N \\
   0 & \text{if } \ell \nmid N
  \end{cases}    
  \end{equation}

\end{lem}
\proof

Define the symbol:

\begin{equation}
  \label{eq:6f}
    \mu(\ell) :=
  \begin{cases}
    N & \textrm{ if } N \mid \ell \\
    0 & \textrm{ if } N \nmid \ell \\
  \end{cases}
\end{equation}
Then
\begin{align}
  \sum_{j=0}^{N-1} \omega^{\ell j} &= 1 + \omega^\ell + \omega^{2\ell} + \dots + \omega^{\ell(N-1)} \\
  &= \mu(\ell)
\end{align}
which gives
\begin{align}
  \sum_{j,k=0}^{N-1} \omega^{\ell(j-k)} &= \sum_{j=0}^{N-1} \sum_{k=0}^{N-1} \omega^{\ell j}\omega^{-\ell k} \\
                                        &= \sum_{j=0}^{N-1} \omega^{\ell j} ( \sum_{k=0}^{N-1} \omega^{-\ell k} ) \\
                                        &= \sum_{j=0}^{N-1} \omega^{\ell j} \mu(\ell) \\
  &= \mu(\ell)^2
\end{align}
\endproof

The key idea is that $\cF_N(f)$ computes the sum of the diagonal entries $a_{\ell \ell}$ of the coefficient matrix of $f$ such that $\ell$ is a multiple of $N$.

\begin{prop}\label{sec:defin-basic-prop-1:computation_of_F}
 Let $f(z,\bar{z}) = \sum a_{jk} \bar{z}^j z^k \in \CC_[z,\bar{z}]$.
 Then:
 \begin{equation}
   \label{eq:53}
  \mathcal{F}_N( f ) = a_{0,0} + a_{N,N} + a_{2N,2N} + \cdots =  \sum_{\substack{\ell = 0 \\ N \mid \ell}}^m a_{\ell \ell}   
 \end{equation}

\end{prop}
\begin{proof}
\begin{align}
\mathcal{F}_N(f) 
   &= \frac{1}{2\pi} \int_{-\pi}^\pi [ \frac{1}{N^2} \sum_{j,k=0}^{N-1} f(\omega^j e^{i\theta},\overline{\omega^k e^{i\theta}} ) ] d\theta \\ 
   &= \frac{1}{N^2} \sum_{j,k=0} [ \frac{1}{2\pi} \int_{-\pi}^\pi f(\omega^j e^{i\theta}, \overline{\omega^k d^{i\theta}}) d\theta ] \\
   &= \frac{1}{N^2} \sum_{j,k=0}^{N-1} [ \sum_{\ell = 0}^m a_{\ell \ell} \omega^{\ell(j-k)} ] \\
   &= \sum_{\ell=0}^m a_{\ell \ell} [ \frac{1}{N^2} \sum_{j,k=0}^{N-1} \omega^{\ell(j-k)} ] \\
   &= \sum_{\substack{\ell = 0 \\ N \mid \ell}}^m a_{\ell \ell}
\end{align} 
where we use Lemma \ref{cha:main-results:diagonals-lemma} and Lemma \ref{cha:main-results:circulant-average-lemma}.
\end{proof}

\begin{cor} $\cF_N$ satisfies the following properties:
  \begin{enumerate}
  \item $\cF_N( z^N \bar{z}^N f(z,\bar{z})) = \cF_N(f(z,\bar{z}))$ for all $f \in \CC_h[z,\bar{z}]$
  \item $\cF_N$ is a $\CC$-linear map $\CC[z,\bar{z}] \rightarrow \CC$
  \item $\cF_N$ is an $\RR$-linear map $\CC_h[z,\bar{z}] \rightarrow \RR$
  \end{enumerate}
\end{cor}
\proof
Follows from Proposition \ref{sec:defin-basic-prop-1:computation_of_F}.
\endproof

\begin{prop}\label{sec:defin-basic-prop:functional_positivity} 
 Let $f \in \CC_h[z,\bar{z}]$.
 Suppose $\text{Gram}(f)[e^{i\theta},\omega e^{i\theta},\dots,\omega^{N-1}e^{i\theta}] \geq 0$ for all $\theta \in [0,2\pi]$.
 Then:

 \begin{equation}
   \label{eq:7f}
  \mathcal{F}_N(f(z,\bar{z})) \geq 0   
 \end{equation}
 
 Furthermore:
 \begin{equation}
   \label{eq:8f}
  \mathcal{F}_N( |h(z)|^2 f(z,\bar{z}) ) \geq 0   
 \end{equation}

 for all holomorphic polynomials $h(z) \in \CC[z]$.
\end{prop}
\begin{proof}

The average of the entries of a positive matrix is positive.  Thus:
\begin{equation}
  \label{eq:9f}
  \frac{1}{N^2} \sum_{j,k=0}^{N-1} f( \omega^j e^{i\theta} , \overline{\omega^k e^{i\theta}} )  \geq 0 \qquad \text{ for all } \theta \in [0,2\pi]
\end{equation}

The average of positive numbers is positive.  Thus:
\begin{equation}
  \label{eq:10f}
  \mathcal{F}_N(f) :=
  \int_{-\pi}^\pi [ \frac{1}{N^2} \sum_{j,k=0}^{N-1} f( \omega^j e^{i\theta} , \overline{\omega^k e^{i\theta}} ) ] \frac{d\theta}{2\pi} \geq 0  
\end{equation}

For a Hermitian polynomial $g \in \CC_h[z,\bar{z}]$, for ease of notation, let $\text{Gram}(g)$ denote the matrix $\text{Gram}(g)[e^{i\theta},\omega e^{i\theta},\dots,\omega^{N-1}e^{i\theta}] $.
We know $\text{Gram}(|h(z)|^2) \geq 0$ for all holomorphic polynomials $h(z) \in \CC[z]$.
Then
\begin{equation}
  \label{eq:11f}
 \text{Gram}(|h(z)|^2 f(z,\bar{z}) ) = \text{Gram}(|h(z)|^2) \circ \text{Gram}(f(z,\bar{z})) \geq 0  
\end{equation}
by the Schur product theorem (where $\circ$ denotes the Hadamard product).

Then $\mathcal{F}_N(|h(z)|^2f(z,\bar{z})) \geq 0$ follows from the first part of the proof.
\end{proof}

\subsection{Representing Matrix Products with $\cF_N$}

Given $v,w \in \CC^N$ and $A \in M_N$, our goal is to construct a polynomial $f$ such that $\cF_N(f) = v^\ast A w$.
Let us expand the expression $v^\ast A w$ so that we may recognize its appearance later.
 Write
 \[ v = \begin{bmatrix} v_0 & \cdots & v_{N-1} \end{bmatrix}^T \]
 \[ w = \begin{bmatrix} w_0 & \cdots & w_{N-1} \end{bmatrix}^T \]
 \[ A = [ a_{jk} ]_{j,k=0}^{N-1} \]
 Then
 \begin{equation}
   \label{eq:54}
  v^\ast A w = \sum_{j,k=0}^{N-1} a_{jk} \bar{v}_j w_k   
 \end{equation}

The following example demonstrates the construction for the case $N=2$.

\begin{ex}
 Let $v,w \in \CC^2$ and $A \in \CC^{2 \times 2}$.  Write:
 \[ v = \begin{bmatrix} v_0 \\ v_1 \end{bmatrix} \qquad
    w = \begin{bmatrix} w_0 \\ w_1 \end{bmatrix} \qquad
    A = \begin{bmatrix} a_{00} & a_{01} \\ a_{10} & a_{11} \end{bmatrix}
 \]
 Define
 \[  f(z,\bar{z}) = \sum_{j,k=0}^1 a_{jk} \bar{z}^j z^k \]
 \[  \tilde{v}(z) = v_0 z^2 + v_1 z \]
 \[  \tilde{w}(z) = w_0 z^2 + w_1 z \]
 Then
 \[
   \overline{ \tilde{v}(z)} \tilde{w}(z) f(z,\bar{z}) =
   a_{00} \bar{v}_0 w_0 z^2 \bar{z}^2 + za_{01}\bar{v}_0w_1 z \bar{z}^2 + \bar{z} a_{10} \bar{v}_1 w_0 z^2 \bar{z} + z\bar{z} a_{11} \bar{v}_1 w_1 z \bar{z}
 \]  
 Then
 \[
  \mathcal{F}_2( \overline{ \tilde{v}(z) } \tilde{w}(z) f(z,\bar{z} )) = a_{00} \bar{v}_0 w_0 + a_{01} \bar{v}_0 w_1 + a_{10} \bar{v}_1 w_0 + a_{11}\bar{w}_1 w_1 = v^\ast A w
 \]
 by Proposition \ref{sec:defin-basic-prop-1:computation_of_F}.

 Written in matrix notation: (consider the coefficient of $z^2 \bar{z}^2$ before applying $\cF_2$)
 \begin{equation}
   \label{eq:55}
   \cF_2 (
   \begin{bmatrix} z^2 \\ z \end{bmatrix}^\ast
   \begin{bmatrix} \bar{v}_0 \\ \bar{v}_1 \end{bmatrix}
   \begin{bmatrix} \bar{w}_0 \\ \bar{w}_1 \end{bmatrix}^\ast
   \begin{bmatrix} z^2 \\ z \end{bmatrix}
   \begin{bmatrix} 1 \\ z \end{bmatrix}^\ast
   \begin{bmatrix} a_{00} & a_{01}\\
   a_{10} & a_{11} \end{bmatrix}
 \begin{bmatrix} 1 \\ z \end{bmatrix})
 =
 \begin{bmatrix} v_0 \\ v_1 \end{bmatrix}^\ast
    \begin{bmatrix} a_{00} & a_{01}\\
      a_{10} & a_{11} \end{bmatrix}
     \begin{bmatrix} w_0 \\ w_1 \end{bmatrix}
 \end{equation}
\end{ex}

We generalize the previous example to obtain the desired construction:

\begin{prop}
  Let:
  \[ v = \begin{bmatrix} v_0 & \cdots & v_{N-1} \end{bmatrix}^T \in \CC^N   \]
  \[ w = \begin{bmatrix} w_0 & \cdots & w_{N-1} \end{bmatrix}^T \in \CC^N   \]
  \[ A = [ a_{jk} ]_{j,k=0}^{N-1} \in \CC^{N\times N} \]
  Define:
  \begin{equation}
    \label{eq:14f}
  \tilde{v}(z) = \sum_{j=0}^{N-1} v_j z^{N-j} \in \CC[z]    
  \end{equation}
  \begin{equation}
    \label{eq:21f}
   \tilde{w}(z) = \sum_{j=0}^{N-1} w_j z^{N-j} \in \CC[z]
  \end{equation}

  \begin{equation}
    \label{eq:24f}
   f(z,\bar{z}) = \sum_{j,k=0}^{N-1} a_{jk} \bar{z}^j z^k \in \CC[z,\bar{z}]    
  \end{equation}
  Then:
  \begin{equation}
    \label{eq:12f}
   \mathcal{F}_N( \overline{ \tilde{v}(z) } \tilde{w}(z) f(z,\bar{z} )) = v^\ast A w     
  \end{equation}
  Furthermore, for integers $s,t \geq 0$ we have:
 \begin{equation}
   \label{eq:13f}
   \mathcal{F}_N( \bar{z}^{Ns} z^{Nt} \overline{ \tilde{v}(z) } \tilde{w}(z) f(z,\bar{z} )) = \delta_{st} v^\ast A w    
 \end{equation}

\end{prop}
\begin{proof}
 We have
 \[
  \overline{\tilde{v}(z)} \tilde{w}(z) = \sum_{j,k=0}^{N-1} \bar{v}_j w_k \bar{z}^{N-j} z^{N-k}
 \]
 and
 \[
  f(z,\bar{z}) = \sum_{c,d=0}^{N-1} a_{cd} \bar{z}^c z^d
 \]
 Thus
 \[
  \overline{\tilde{v}(z)} \tilde{w}(z) f(z,\bar{z}) = \sum_{j,k,c,d=0}^{N-1} a_{cd} \bar{v}_j w_k \bar{z}^{N-j+c} z^{N-k+d}
 \]
Use the $\CC$-linearity of $\mathcal{F}_N$:
\[
 \mathcal{F}_N(\overline{\tilde{v}(z)} \tilde{w}(z) f(z,\bar{z})) = \sum_{j,k,c,d=0}^{N-1} a_{cd} \bar{v}_j w_k \mathcal{F}_N(\bar{z}^{N-j+c} z^{N-k+d}) 
\]
Since $0\leq j,k,c,d \leq N-1$, applying Proposition we get 
\[
\mathcal{F}_N(\bar{z}^{N-j+c} z^{N-k+d}) = \delta_{jc}\delta_{kd}
\]
Therefore only terms with $j=c$ and $k=d$ will survive $\mathcal{F}_N$:
\[
 \mathcal{F}_N(\overline{\tilde{v}(z)} \tilde{w}(z) f(z,\bar{z})) = \sum_{j,k=0}^{N-1} a_{jk}\bar{v}_j w_k = v^\ast A w
\]

In order to prove the second part, for a Hermitian polynomial $f$, let $\Mon(f) \subset \NN \times \NN$ denote the monomial support (the set of monomials corresponding to nonzero coefficients of $f$).
Note $\Mon(fg) = \Mon(f) + \Mon(g)$.
Say a monomial in $\NN \times \NN$ is \emph{$N$-divisible} if both components are divisible by $N$ and say the monomial is \emph{diagonal} if both entries are equal.
We know that $\cF_N(f)$ is the sum of the coefficients corresponding to monomials of $f$ which are both diagonal and $N$-divisible.

For integers $a,b \in \ZZ$ let $[a,b] := \{ j \in \ZZ | a \leq j \leq b \}$.
Then:
\begin{equation}
  \label{eq:57}
  \Mon( \overline{ \tilde{v}(z)} \tilde{w}(z) ) = [1,N] \times [1,N]
\end{equation}
\begin{equation}
  \label{eq:58}
  \Mon ( f(z,\bar{z}) = [0,N-1] \times [0,N-1]
\end{equation}
Hence
\begin{equation}
  \label{eq:59}
  \Mon (  \overline{ \tilde{v}(z)} \tilde{w}(z) f(z,\bar{z}) ) =[1,2N-1]\times[1,2N-1]
\end{equation}
whose only diagonal $N$-divisible element is $(N,N)$.

We have
\begin{equation}
  \label{eq:61}
  \Mon (  \overline{ \tilde{v}(z)} \tilde{w}(z) f(z,\bar{z}) \bar{z}^{Ns} z^{Nt} )
  =
  [1+Ns, 2N-1+Ns] \times [ 1+Nt, 2N-1+Nt]
\end{equation}

For convenience, denote $Q(s,t) = Q(s,t)(z,\bar{z}) =  \overline{ \tilde{v}(z)} \tilde{w}(z) f(z,\bar{z}) \bar{z}^{Ns} z^{Nt}$.

As we vary $(s,t) \in \ZZ_{\geq 0} \times \ZZ_{\geq 0}$ we obtain a partition of $\ZZ_{>0} \times \ZZ_{>0}$ by squares $\Mon(Q(s,t))$.  Only the bottom right monomial of each tile is $N$-divisible.
If $s=t$, then the bottom right element of $\Mon(Q(s,t))$ is both $N$-divisible and diagonal.
If $s\neq t$, then no elements of $\Mon (Q(s,t))$ are both $N$-divisible and diagonal.
From these observations the result follows.

\end{proof}


\section{Proof of Main Results}

  We are now in position to utilize all the tools developed so far and show that the matrix positivity conditions imply positivity of the associated block Toeplitz matrix.
  We then prove the theorem characterizing Hermitian sums of squares modulo $(z^N \bar{z}^N-1)$.

\begin{prop}\label{sec:repr-matr-prod:gram-to-toep}
 Let $\omega = e^{2 \pi i / N}$.
 
 Let $f \in \CC_h[z,\bar{z}]$ be trigonometric mod $(z^N \bar{z}^N - 1)$ with data $(A_0,\dots,A_m)$.
 
 Suppose $\text{Gram}(f)[e^{i\theta},\omega e^{i\theta},\dots,\omega^{N-1}e^{i\theta}] \geq 0$ for all $\theta \in [0,2\pi]$.
 
 Then $\text{Toep}(A_0,\dots,A_m) \geq 0$.
\end{prop}
\begin{proof}

Let $v^{(0)},\dots,v^{(m)} \in \CC^N$ be arbitrary.

Write $v^{(j)} = \begin{bmatrix} v_0^{(j)} & \cdots & v_{N-1}^{(j)} \end{bmatrix}^T$ with $v_k^{(j)} \in \CC$.

Let $v = \begin{bmatrix} v^{(0)} & \cdots & v^{(m)} \end{bmatrix}^T \in (\CC^N)^{m+1}$.

It suffices to show $v^\ast \text{Toep}(A_0,\dots,A_m) v \geq 0$.

Observe that
\[
 v^\ast \text{Toep}(A_0,\dots,A_m) v = \sum_{j,k=0}^m v^{(j)\ast} A_{k-j} v^{(k)}
\]

For $j=0,\dots,m$ define $\tilde{v}^{(j)}(z) = \sum_{k=0}^{N-1} v_k^{(j)} z^{N-k} \in \CC[z]$.

Define $v(z) = \sum_{j=0}^m z^{Nj} \tilde{v}^{(j)}(z) \in \CC[z]$.

Let $\psi(z) = \phi_N(z) = \begin{bmatrix} 1 & z & \cdots & z^{N-1} \end{bmatrix}^T$.

Claim: $\mathcal{F}_N( |v(z)|^2 f(z,\bar{z}) ) = v^\ast \text{Toep}(A_0,\dots,A_m) v \geq 0$.

Introduce double-index notation for matrix coefficients of $f$ in order to simplify computations:
 \[
  f(z,\bar{z}) = 
  \begin{bmatrix} \psi(z) \\ z^N \psi(z) \\ \vdots \\ z^{Nm}\psi(z) \end{bmatrix}^\ast
  \begin{bmatrix}
    A_{00} & A_{01} & \cdots & A_{0m} \\
    A_{10} & A_{11} & \cdots & A_{1m} \\
    \vdots & \vdots & \ddots & \vdots \\
    A_{m0} & A_{m1} & \cdots & A_{mm}
  \end{bmatrix}
  \begin{bmatrix} \psi(z) \\ z^N \psi(z) \\ \vdots \\ z^{Nm}\psi(z) \end{bmatrix}
 \]
where $A_{jk} = 0$ if $\min(k,j) \neq 0$ and $A_{jk} = A_{k-j}$ otherwise.

Then:
\[
 f(z,\bar{z}) = \sum_{j,k=0}^m \bar{z}^{jN} z^{kN} \psi(z)^\ast A_{jk} \psi(z)
\]

We have:
\[
 |v(z)|^2 = \sum_{j,k=0}^m \bar{z}^{Nj} z^{Nk} \overline{ \tilde{v}^{(j)}(z) } \tilde{v}^{(k)}(z)
\]
Hence:
\[
 |v(z)|^2 f(z,\bar{z})
 =
 \sum_{j,k,c,d=0}^m \bar{z}^{N(c+j)} z^{N(d+k)} \overline{ \tilde{v}^{(j)}(z) } \tilde{v}^{(k)}(z) \psi(z)^\ast A_{cd} \psi(z)
\]
Then using Proposition \ref{sec:defin-basic-prop-1:computation_of_F} and passing between the double and single index coefficients for $f$:
\begin{align}
\mathcal{F}_N( |v(z)|^2 f(z,\bar{z}) ) &= \sum_{j,k,c,d=0}^m \mathcal{F}_N ( \bar{z}^{N(c+j)} z^{N(d+k)} \overline{ \tilde{v}^{(j)}(z) } \tilde{v}^{(k)}(z) \psi(z)^\ast A_{cd} \psi(z) ) \\ 
&= \sum_{j,k,c,d=0}^m \delta_{c+j,d+k} v^{(j)\ast} A_{cd} v^{(k)} \\
&= \sum_{\substack{j,k,c,d=0 \\ c+j=d+k}}^m v^{(j)\ast} A_{cd} v^{(k)} \\
&= \sum_{\substack{j,k,c,d=0 \\ j-k=d-c}}^m v^{(j)\ast} A_{d-c} v^{(k)} \\
&= \sum_{j,k=0}^m v^{(j)\ast} A_{j-k} v^{(k)} \\
&= v^\ast \text{Toep}(A_0,\dots,A_m)^T v 
\end{align}

By Proposition \ref{sec:defin-basic-prop:functional_positivity}  we have
\[
 v^\ast \text{Toep}(A_0,\dots,A_m)^T v \geq 0 \qquad \forall v \in (\CC^N)^{m+1}
\]
Therefore $\text{Toep}(A_0,\dots,A_m)^T \geq 0$, and hence $\text{Toep}(A_0,\dots,A_m) \geq 0$.

\end{proof}

\begin{thm}\label{MAIN-THEOREm}
 Let $\omega = e^{2 \pi i / N}$.
 
 Let $f \in \CC_h[z,\bar{z}]$ be trigonometric mod $(z^N \bar{z}^N - 1)$ with data $(A_0,\dots,A_m)$.
 
 The following conditions are equivalent:
 \begin{enumerate}
  \item $f \in \Sigma^2_h + (z^N \bar{z}^N - 1)$
  \item $\text{Gram}(f)[e^{i\theta},\omega e^{i\theta},\dots,\omega^{N-1}e^{i\theta}] \geq 0$ for all $\theta \in [0,2\pi]$
  \item $\text{Toep}(A_0,\dots,A_m) \geq 0$
 \end{enumerate}
\end{thm}
\begin{proof}
  
( (i) $\implies$ (ii) ) By Lemma \ref{MATRIX-POSITIVITY-CONDITIONS}.

( (ii) $\implies$ (iii) ) By Proposition \ref{sec:repr-matr-prod:gram-to-toep}.

( (iii) $\implies$ (i) ) By Corollary \ref{sec:operator-valued-rf:OPERATOR-RF} and the conclusion of Section \ref{sec:block-trace-param}.

\end{proof}


\bibliographystyle{plain}
\bibliography{references}

\end{document}